\documentclass[11pt,oneside,reqno]{amsart}
\usepackage{stylesheet}
\newcommand{\R}{\mathbb{R}}
\newcommand{\Z}{\mathbb{Z}}
\newcommand{\mcR}{\mathcal{R}}

\title{Kirby diagrams for an infinite family of exotic $\R^4$'s}

\author{Siddharth Shrivastava}
\address{The University of Melbourne, Melbourne VIC, Australia}
\email{ssshrivastav@student.unimelb.edu.au}

\begin{document}

\begin{abstract}
    Eli, Hom, and Lidman showed that the manifolds produced by attaching the simplest positive Casson handle $CH^+$ to a slice disc complement of the ribbon knot $T_{2,n}\#T_{2,-n}$ for $n\ge3$ and odd, and removing the boundary, form a countably infinite family of exotic $\R^4$'s. They provided a Kirby diagram for the case $n=3$.
    In this short note, we extend this for $n\ge3$ and odd, and provide Kirby diagrams for two such families of exotic $\R^4$'s, which are then shown to be equivalent. We then generalise these diagrams to a family of exotic $\R^4$'s built using ribbon disc complements of the pretzel knots $P(n,-n,2k)$.
\end{abstract}

\maketitle

\section{Introduction}
Exotic $\R^4$'s are smooth manifolds that are homeomorphic but not diffeomorphic to the standard $\R^4$. This is a phenomenon exclusive to dimension $4$, because for all other $n\ne4$, any smooth manifold homeomorphic to $\R^n$ is also diffeomorphic to the standard $\R^n$ \cite{stallings}. It is known that uncountably many exotic $\R^4$'s exist \cite{taubes}. One way to construct exotic $\R^4$'s involves attaching a Casson handle to a ribbon disc complement and removing the boundary, giving what are called \emph{ribbon} $\R^4$'s (see \cite{demichelisfreedman} and \cite[Chapter 14]{kirby}). 

Eli, Hom, and Lidman \cite{elihomlidman} provided a way to distinguish exotic $\R^4$'s that are in the form of \emph{slice} $\R^4$'s, which are built by attaching a Casson handle to a \emph{slice} disc complement and removing the boundary, by using Gadgil's end Floer homology \cite{gadgil}. This allowed them to produce a countably infinite family of exotic $\R^4$'s. Specifically, they showed that the manifolds $\mcR_n$, obtained by attaching the simplest positive Casson handle $CH^+$ to a slice disc complement of the ribbon knot $T_{2,n}\#T_{2,-n}$ (for $n\ge 3$ and odd) and removing the boundary, are pairwise nondiffeomorphic exotic $\R^4$'s \cite[Corollary 1.2]{elihomlidman}. This result does not depend on the choice of slice disc. That is, we can construct each $\mcR_n$ by using any slice disc for the knot $T_{2,n}\#T_{2,-n}$, and the resulting family of manifolds $\{\mcR_{n}\}_{n=3,5,7,...}$ will be a family of pairwise nondiffeomorphic exotic $\R^4$'s. This is due to the fact that the end Floer homology does not depend on the choice of slice disc, but only on the slice knot. 

To obtain a more concrete understanding of the above families of exotic $\R^4$'s, it can be useful to represent them using Kirby diagrams. Eli, Hom, and Lidman provided a Kirby diagram for one such choice of $\mcR_3$, built using a ribbon disc complement for $T_{2,3}\#T_{2,-3}$ \cite[Figure 1]{elihomlidman}. In this short note, we first extend this for $n\ge3$ and odd, and provide the Kirby diagrams for two such families of exotic $\R^4$'s. 

\begin{theorem}
    Let $\mcR_n$ and $\mcR_n'$ be the interiors of the manifolds shown in \cref{fig:Rn1} and \cref{fig:Rn2} respectively for $n\ge3$ and odd. Then $\{\mcR_n\}_{n=3,5,7,...}$ and $\{\mcR_n'\}_{n=3,5,7,...}$ are each families of pairwise nondiffeomorphic exotic $\R^4$'s.
    \label{thm:1}
\end{theorem}

Here, $\mcR_n$ and $\mcR_n'$ are both constructed using the ribbon knot $T_{2,n}\#T_{2,-n}$ and the Casson handle $CH^+$, but their ribbon disc complements are obtained using different ribbon moves. The case of $n=3$ in \cref{fig:Rn1} yields the same diagram provided in \cite[Figure 1]{elihomlidman}. 

\begin{figure}[ht]
    \centering
    \begin{tikzpicture}
        \node[anchor=center] at (0,0){\includegraphics[scale=0.4]{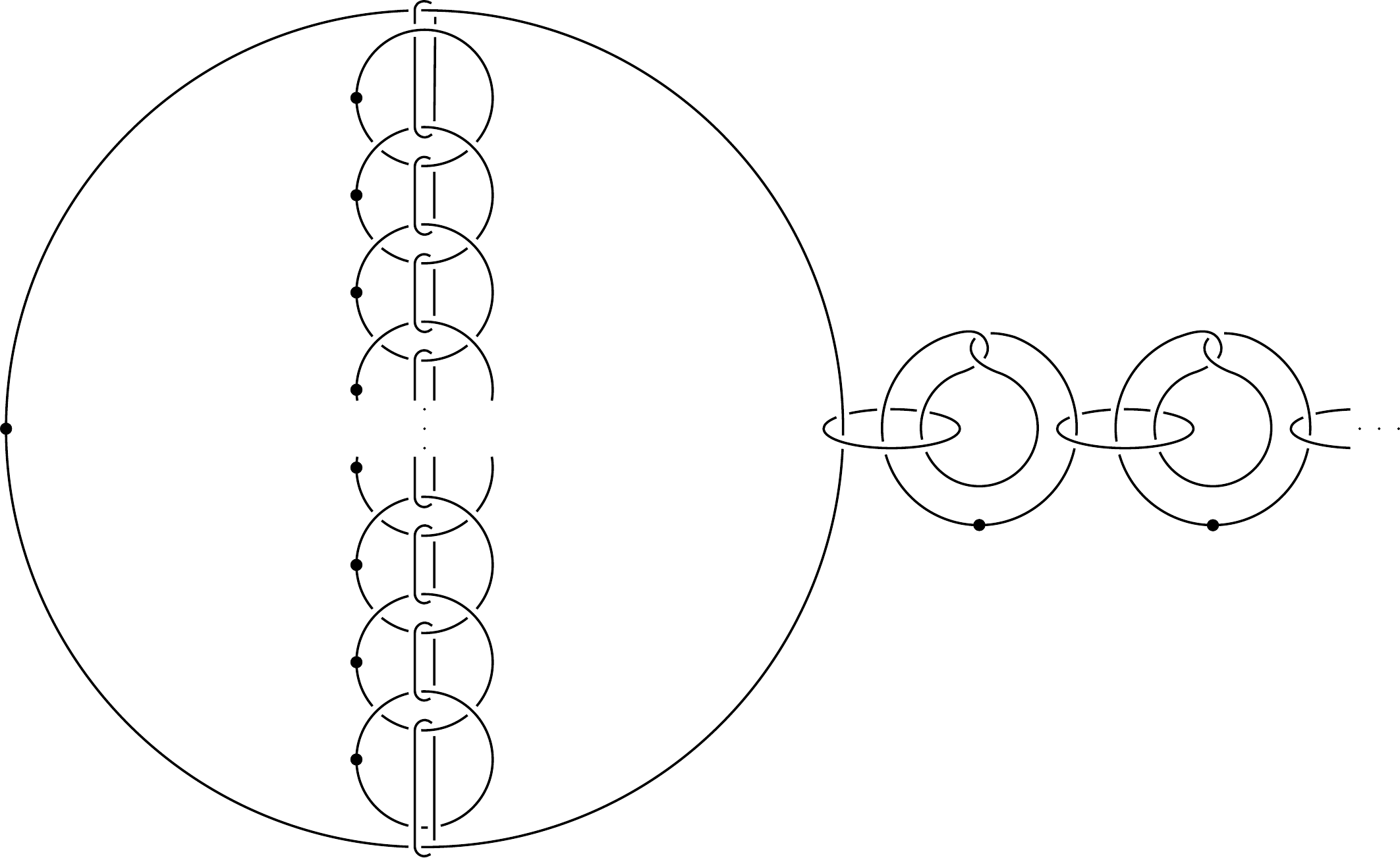}};
        \node at (-2.55, 0.5) {$0$};
        \node at (-2.55, 1.425) {$0$};
        \node at (-2.55, 2.45) {$0$};
        \node at (-2.55, 3.6) {$0$};
        \node at (-2.55, -0.5) {$0$};
        \node at (-2.55, -1.425) {$0$};
        \node at (-2.55, -2.45) {$0$};
        \node at (-2.55, -3.6) {$0$};
        \draw[decorate, decoration={brace}] (-3.8, -4) -- (-3.8, 4);
        \node [align=center] at (-4.6, 0) {$n-1$ \\ dotted \\ circles};
        \node at (1.7, 0.45) {$0$};
        \node at (4.1, 0.45) {$0$};
        \node at (6.5, 0.45) {$0$};
    \end{tikzpicture}
    \caption{A Kirby diagram for $\mcR_n$. The diagram includes a total of $n-1$ dotted circles inside the large dotted circle.}
    \label{fig:Rn1}
\end{figure}
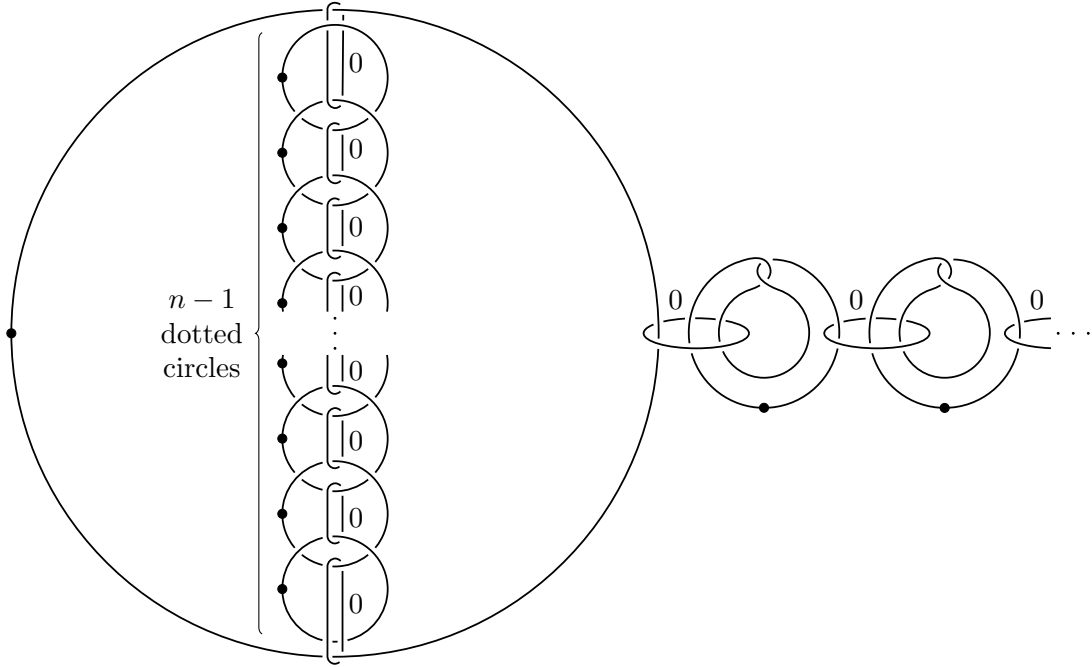

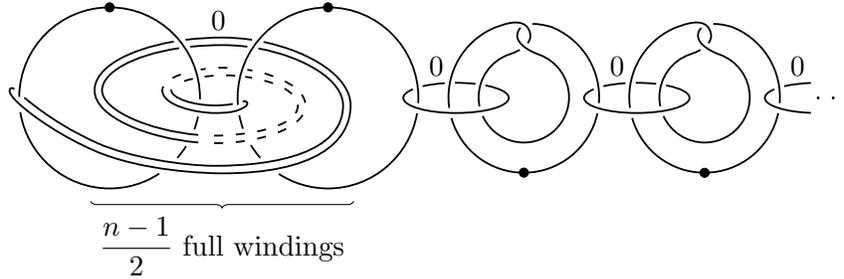
\begin{figure}[ht]
    \centering
    \begin{tikzpicture}
        \node[anchor=center] at (0,0) 
        {\includegraphics[scale=0.4]{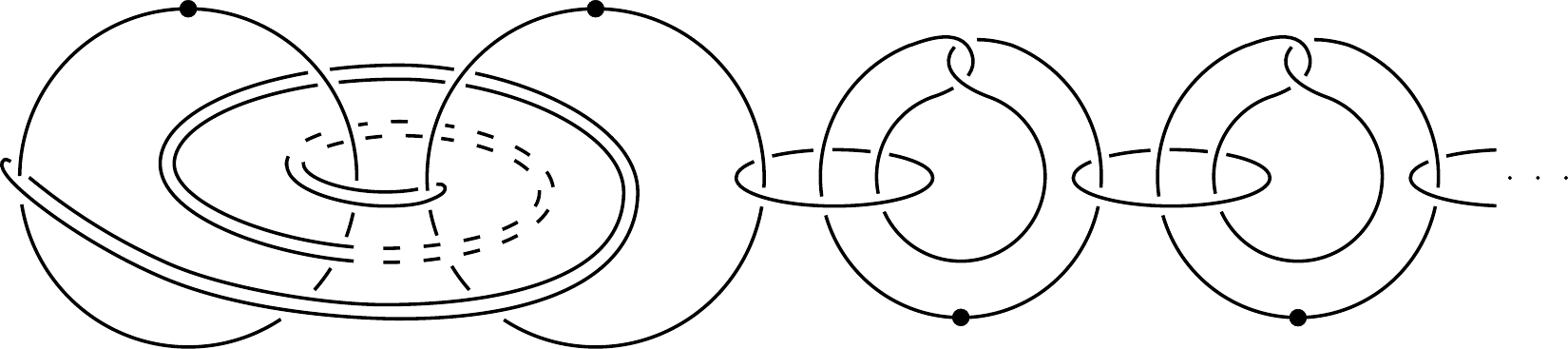}};
        \node at (-2.8,1) {$0$};
        \draw[decorate, decoration={brace}] (-1,-1.4) -- (-4.5,-1.4); 
        \node at (-2.75,-2) {$\dfrac{n-1}{2}$ full windings};
        \node at (0.1,0.4) {$0$};
        \node at (2.5,0.4) {$0$};
        \node at (4.9,0.4) {$0$};
    \end{tikzpicture}
    \caption{A Kirby diagram for $\mcR_n'$. The dashed lines represent $(n-1)/2$ full windings around the inner strands of the two dotted circles.}
    \label{fig:Rn2}
\end{figure}

While at first glance the ribbon disc complements used in the construction of $\mcR_n$ and $\mcR_n'$ (for a fixed $n$) may seem distinct, using Kirby moves, we show that the complements are diffeomorphic. As a result, the two families of Kirby diagrams in \cref{thm:1} represent the same family of exotic~$\R^4$'s.

\begin{theorem}
    The manifolds $\mcR_n$ and $\mcR_n'$ are diffeomorphic for all $n\ge3$ and odd. Thus, $\{\mcR_n\}_{n=3,5,7,...}$ and $\{\mcR_n'\}_{n=3,5,7,...}$ are the same family of exotic $\R^4$'s. 
    \label{thm:2}
\end{theorem}

\begin{remark}
    \cref{fig:Rn1Alt} shows an alternative way to draw the Kirby diagram in \cref{fig:Rn1}. This is obtained by isotoping the leftmost strand of the large dotted circle in \cref{fig:Rn1} over the $n-1$ dotted circles, and then attaching the Casson handle $CH^+$.
    \label[remark]{rem:altdiagram}
\end{remark}

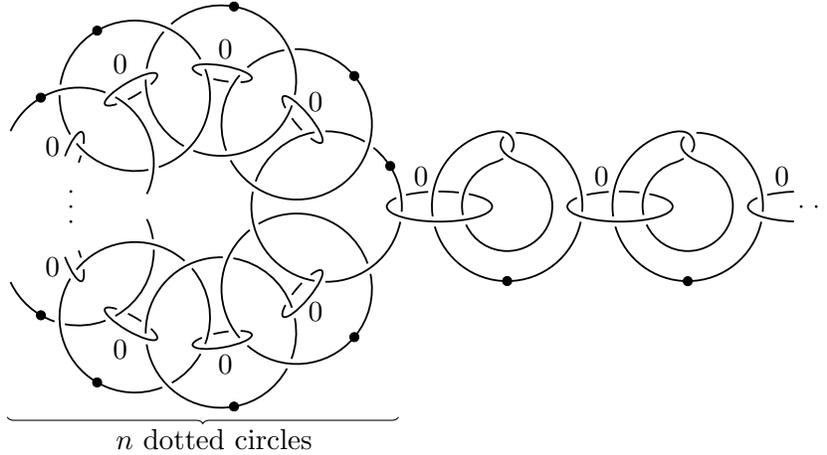
\begin{figure}[ht]
    \centering
    \begin{tikzpicture}
        \node[anchor=center] at (0,0) 
        {\includegraphics[scale=0.4]{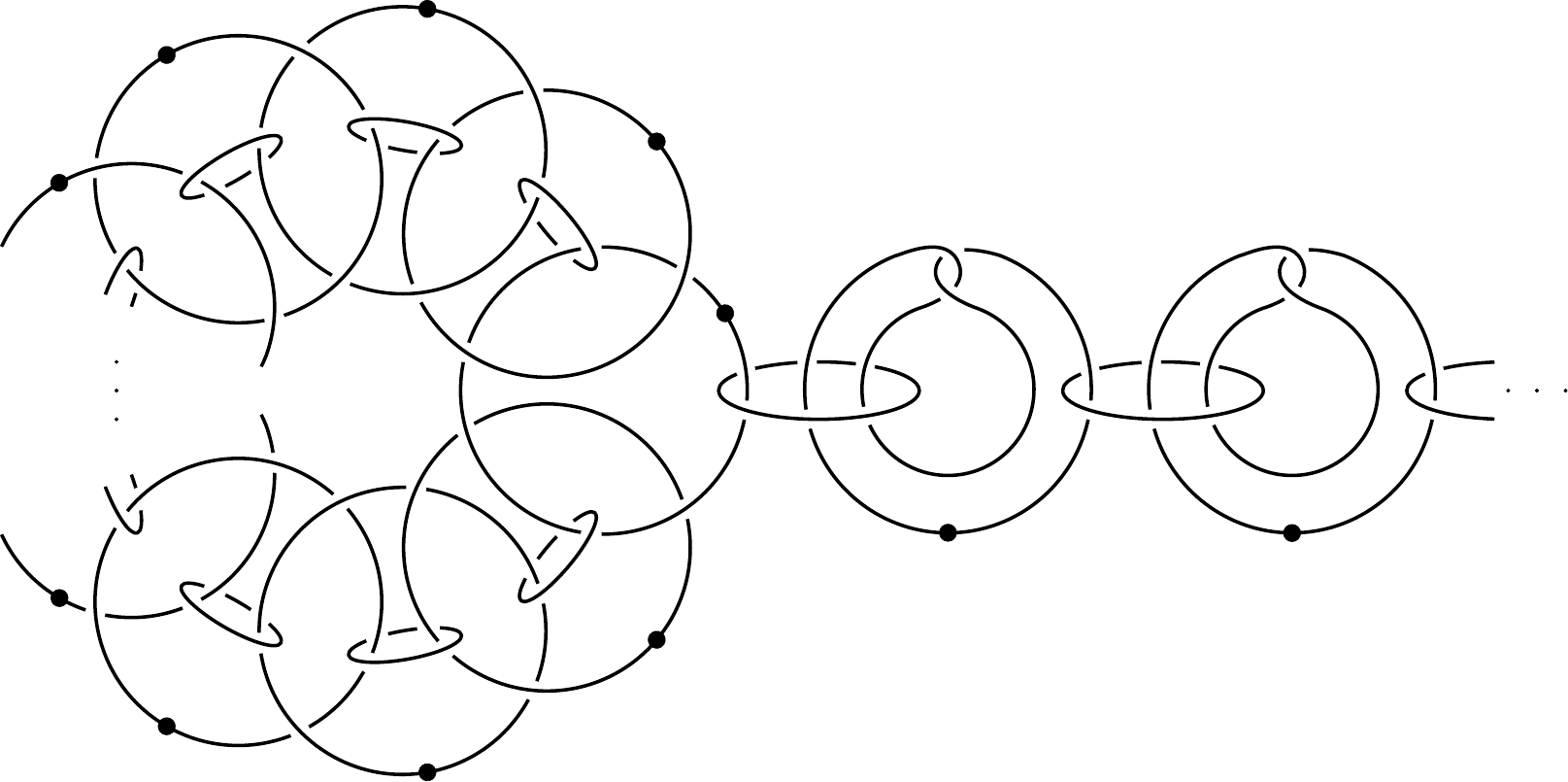}};
        \node at (-1.4,1.4) {$0$};
        \node at (-2.6,2.1) {$0$};
        \node at (-4,1.9) {$0$};
        \node at (-4.9,0.8) {$0$};
        \node at (-1.4,-1.4) {$0$};
        \node at (-2.6,-2.1) {$0$};
        \node at (-4,-1.9) {$0$};
        \node at (-4.9,-0.8) {$0$};
        \draw[decorate, decoration={brace}] (-0.3,-2.8) -- (-5.5,-2.8); 
        \node at (-2.75,-3.1) {$n$ dotted circles};
        \node at (0,0.4) {$0$};
        \node at (2.4,0.4) {$0$};
        \node at (4.8,0.4) {$0$};
    \end{tikzpicture}
    \caption{An equivalent Kirby diagram to \cref{fig:Rn1}. This can be obtained from \cref{fig:Rn1} by isotoping the leftmost strand of the large dotted circle over the $n-1$ dotted circles, and then attaching the Casson handle $CH^+$.}
    \label{fig:Rn1Alt}
\end{figure}

The Kirby diagram in \cref{fig:Rn2} can be generalised by adding twists in the parallel strands of the $2$-handle, as shown in \cref{fig:Rnk}. We will show that such a manifold, with $k$ full twists in the $2$-handle, is also an exotic $\R^4$ and can be constructed by attaching the Casson handle $CH^+$ to a ribbon disc complement of the pretzel knot $P(n,-n,2k)$. We thank Tye Lidman for suggesting this generalisation. 

\begin{theorem}
    Let $\mcR_{n,k}$ be the interior of the manifold shown in \cref{fig:Rnk} for $n\ge3$ and odd, and $k\in\Z$. Then $\mcR_{n,k}$ is an exotic $\R^4$. Furthermore, if $n\ne n'$, then $\mcR_{n,k}$ and $\mcR_{n',k'}$ are not diffeomorphic for all $k,k'\in\Z$.
    \label{thm:3}
\end{theorem}

We prove this by considering the knot Floer homology of the pretzel knot $P(n,-n,2k)$ and by using the results of Eli, Hom, and Lidman \cite[Theorem 1.1]{elihomlidman}.
However, we cannot use the same tools to distinguish $\mcR_{n,k}$ and $\mcR_{n',k'}$ if $n=n'$. This is because the knot $P(n,-n,2k)$ has the same knot Floer homology for all $k\in\Z$.

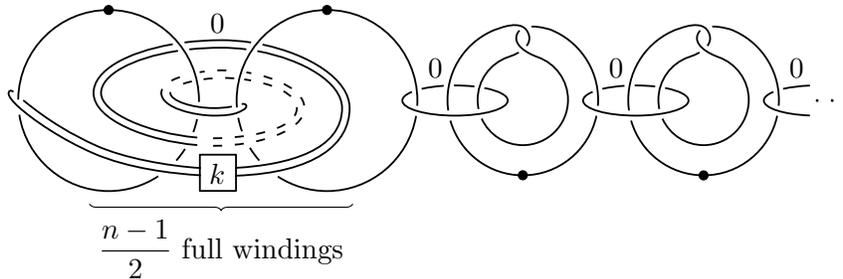
\begin{figure}[ht]
    \centering
    \begin{tikzpicture}
        \node[anchor=center] at (0,0) 
        {\includegraphics[scale=0.4]{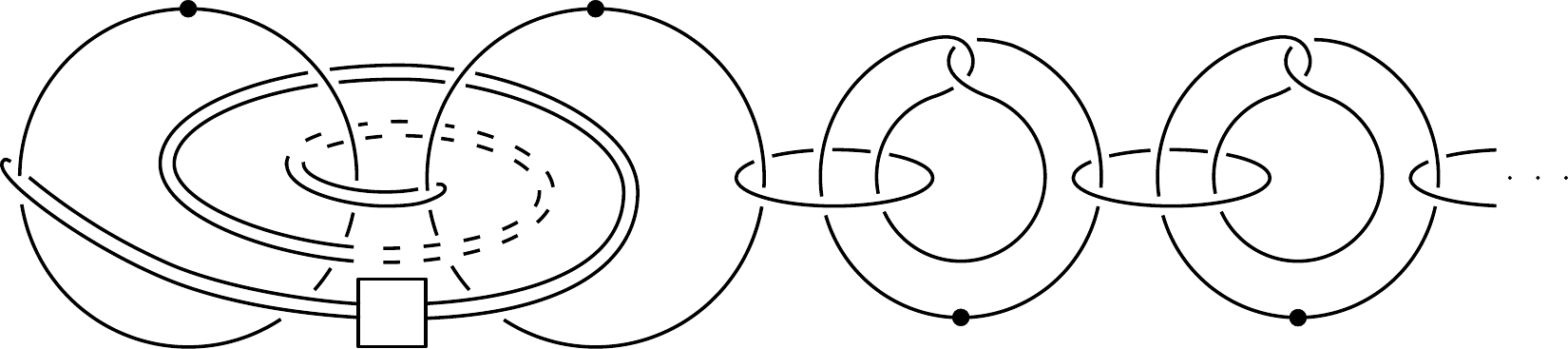}};
        \node at (-2.8,1) {$0$};
        \node at (-2.8,-1) {$k$};
        \draw[decorate, decoration={brace}] (-1,-1.4) -- (-4.5,-1.4); 
        \node at (-2.75,-2) {$\dfrac{n-1}{2}$ full windings};
        \node at (0.1,0.4) {$0$};
        \node at (2.5,0.4) {$0$};
        \node at (4.9,0.4) {$0$};
    \end{tikzpicture}
    \caption{A Kirby diagram for $\mcR_{n,k}$. The box represents $k$ full twists in the two parallel strands of the $2$-handle.}
    \label{fig:Rnk}
\end{figure}

\subsection*{Note}
We have been informed that the Kirby diagrams in \cref{fig:Rn1} were also independently obtained by the authors of \cite{elihomlidman} in unpublished work. 

\subsection*{Acknowledgments} 
This work was done during the 2026 Vacation Scholarship Program with the School of Mathematics and Statistics at The University of Melbourne. I am deeply grateful to my supervisor, Arunima Ray, for her valuable guidance, many helpful discussions, and for thoughtful feedback on earlier drafts of this paper. I would like to extend my sincere thanks to Sean Eli, Jennifer Hom, and Tye Lidman for insightful comments, suggestions, and a discussion on an earlier draft. I also thank the coordinators of the Vacation Scholarship Program for their support and the opportunity to work on this research project.

\section{Proofs}

\begin{proof}[Proof of \cref{thm:1}]
    We first construct the family $\{\mcR_n\}_{n=3,5,7,...}$ represented by \cref{fig:Rn1}. We will only show the explicit construction for $\mcR_5$, and then generalise this to all $n\ge3$ and odd. \cref{fig:T1Knots_a} is the torus knot $T_{2,5}$. The knot $T_{2,-5}$ is the mirror image of $T_{2,5}$, and so the connected sum $T_{2,5}\#T_{2,-5}$ is the knot shown in \cref{fig:T1Knots_b} and \cref{fig:T1Knots_c}. Since torus knots are invertible \cite[Proposition 3.27]{burdezieschang}, we do not need to consider orientations when taking the connected sum. 
    
    \begin{figure}[ht]
        \centering
        \begin{subfigure}[b]{0.25\textwidth}
            \centering
            \begin{tikzpicture}
                \node[anchor=center] at (0,0) {\includegraphics[scale=0.4]{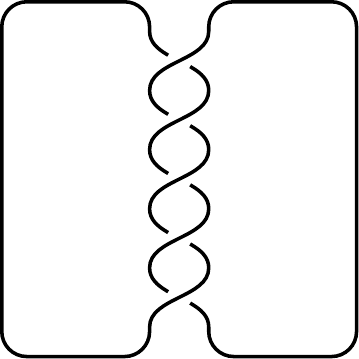}};
            \end{tikzpicture}
            \caption{}
            \label{fig:T1Knots_a}
        \end{subfigure}
        \begin{subfigure}[b]{0.40\textwidth}
            \centering
            \begin{tikzpicture}
                \node[anchor=center] at (0,0) {\includegraphics[scale=0.4]{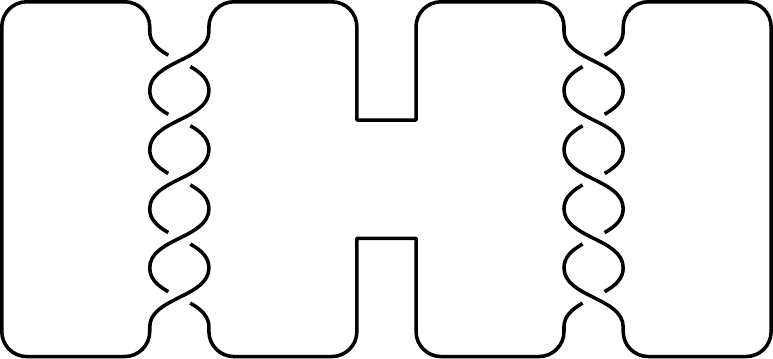}};
            \end{tikzpicture}
            \caption{}
            \label{fig:T1Knots_b}
        \end{subfigure}
        \begin{subfigure}[b]{0.25\textwidth}
            \centering
            \begin{tikzpicture}
                \node[anchor=center] at (0,0) {\includegraphics[scale=0.4]{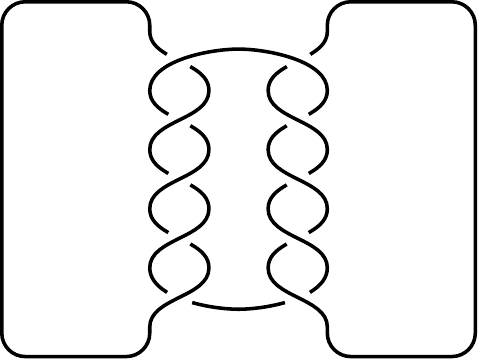}};
            \end{tikzpicture}
            \caption{}
            \label{fig:T1Knots_c}
        \end{subfigure}
        \caption{(a) The torus knot $T_{2,5}$. (b) and (c) The connected sum $T_{2,5}\#T_{2,-5}$.}
        \label{fig:T1Knots}
    \end{figure}

    A ribbon disc for the knot $T_{2,5}\#T_{2,-5}$ can now be obtained by doing four \emph{ribbon moves} (see \cite[Section 6.2]{gompfstipsicz}) as shown in \cref{fig:T1Ribbon_a}. The fine bands represent the ribbon moves used to obtain the specific ribbon disc, and the dot indicates that we are taking the ribbon disc's complement \cite[Section 6.2]{gompfstipsicz}. To get the Kirby diagram of the ribbon disc complement, we follow the procedure described in \cite[Section 1.4]{akbulut}. That is, we perform a band sum and add a $0$-framed $2$-handle along each ribbon move, as shown in \cref{fig:T1Ribbon_b}. This turns the knot into five unlinked dotted circles. \cref{fig:T1Ribbon_c} is equivalent to \cref{fig:T1Ribbon_b}, and represents the Kirby diagram for a ribbon disc complement of $T_{2,5}\#T_{2,-5}$. 

    \begin{figure}[ht]
        \centering
        \begin{subfigure}[b]{0.30\textwidth}
            \centering
            \begin{tikzpicture}
                \node[anchor=center, inner ysep=15pt] at (0,0) {\includegraphics[scale=0.4]{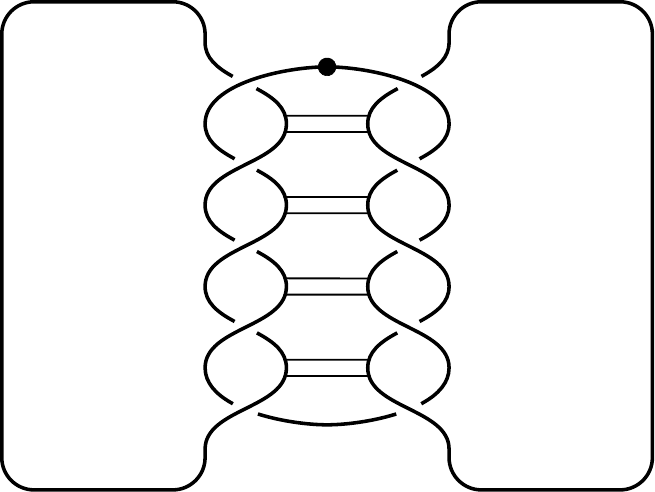}};
            \end{tikzpicture}
            \caption{}
            \label{fig:T1Ribbon_a}
        \end{subfigure}
        \begin{subfigure}[b]{0.30\textwidth}
            \centering
            \begin{tikzpicture}
                \node[anchor=center, inner ysep=15pt] at (0,0) {\includegraphics[scale=0.4]{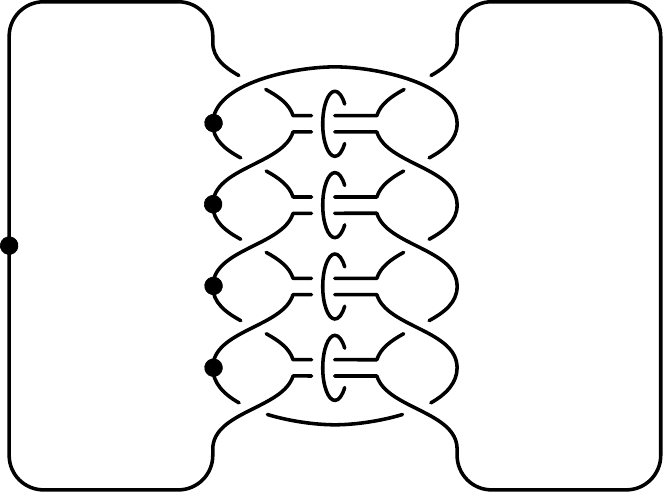}};
            \end{tikzpicture}
            \caption{}
            \label{fig:T1Ribbon_b}
        \end{subfigure}
        \begin{subfigure}[b]{0.30\textwidth}
            \centering
            \begin{tikzpicture}
                \node[anchor=center] at (0,0) {\includegraphics[scale=0.4]{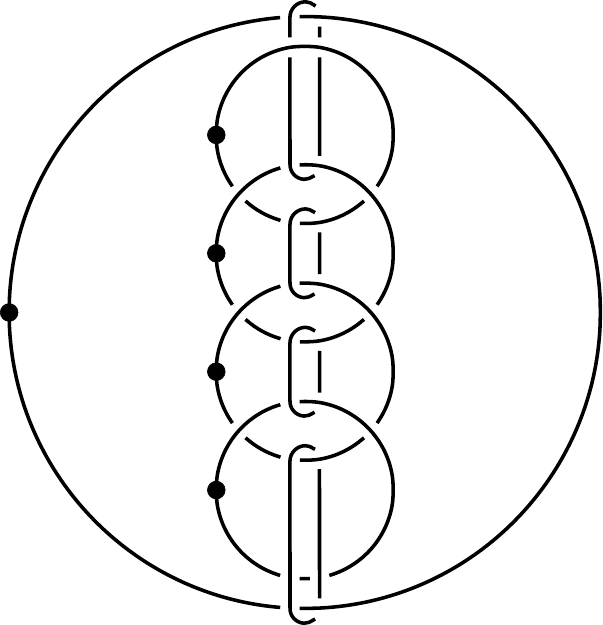}};
                \node at (0.3,0.4) {$0$};
                \node at (0.3,1.4) {$0$};
                \node at (0.3,-0.4) {$0$};
                \node at (0.3,-1.4) {$0$};
            \end{tikzpicture}
            \caption{}
            \label{fig:T1Ribbon_c}
        \end{subfigure}
        \caption{Ribbon moves for a ribbon disc complement of the knot $T_{2,5}\#T_{2,-5}$. All $2$-handles in (b) are $0$-framed.}
        \label{fig:T1Ribbon}
    \end{figure}
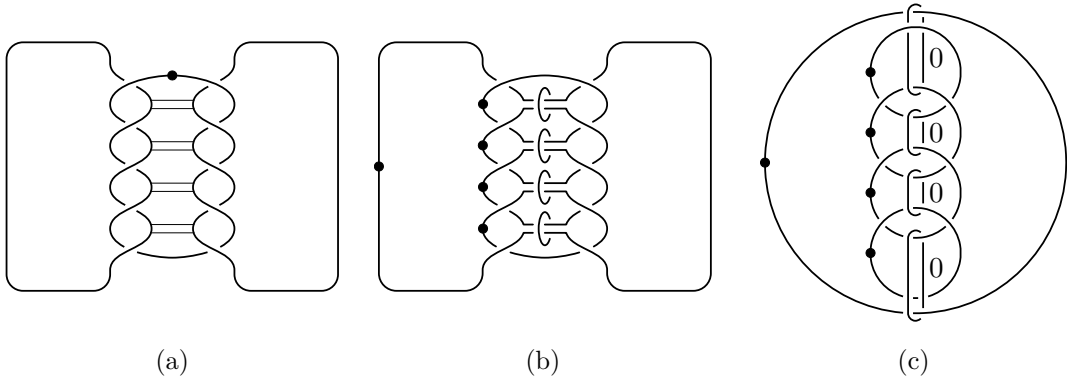

    We will now obtain an exotic $\R^4$ by attaching a Casson handle to the ribbon disc complement shown in \cref{fig:T1Ribbon_c}. Casson handles are built using layers of self-plumbed $2$-handles. In general, Casson handles can have positive and negative self-plumbings, and also branching \cite[Section 6.1]{gompfstipsicz}. We only consider the simplest positive Casson handle $CH^+$, which has a single positive self-plumbing at each stage. For some other Casson handles, see \cref{rem:cassonhandles}. 

    To construct the exotic $\R^4$, we now attach the Casson handle $CH^+$ along a $0$-framed meridian of the ribbon disc that is removed (as described in \cite{elihomlidman} and \cite{gompfstipsicz}). This attachment is shown in \cref{fig:R51}. By \cite[Corollary 1.2]{elihomlidman}, the manifold $\mcR_5$, obtained from \cref{fig:R51} by removing the boundary, is an exotic $\R^4$.

    \begin{figure}[ht]
        \centering
        \begin{tikzpicture}
            \node[anchor=center] at (0,0) 
            {\includegraphics[scale=0.4]{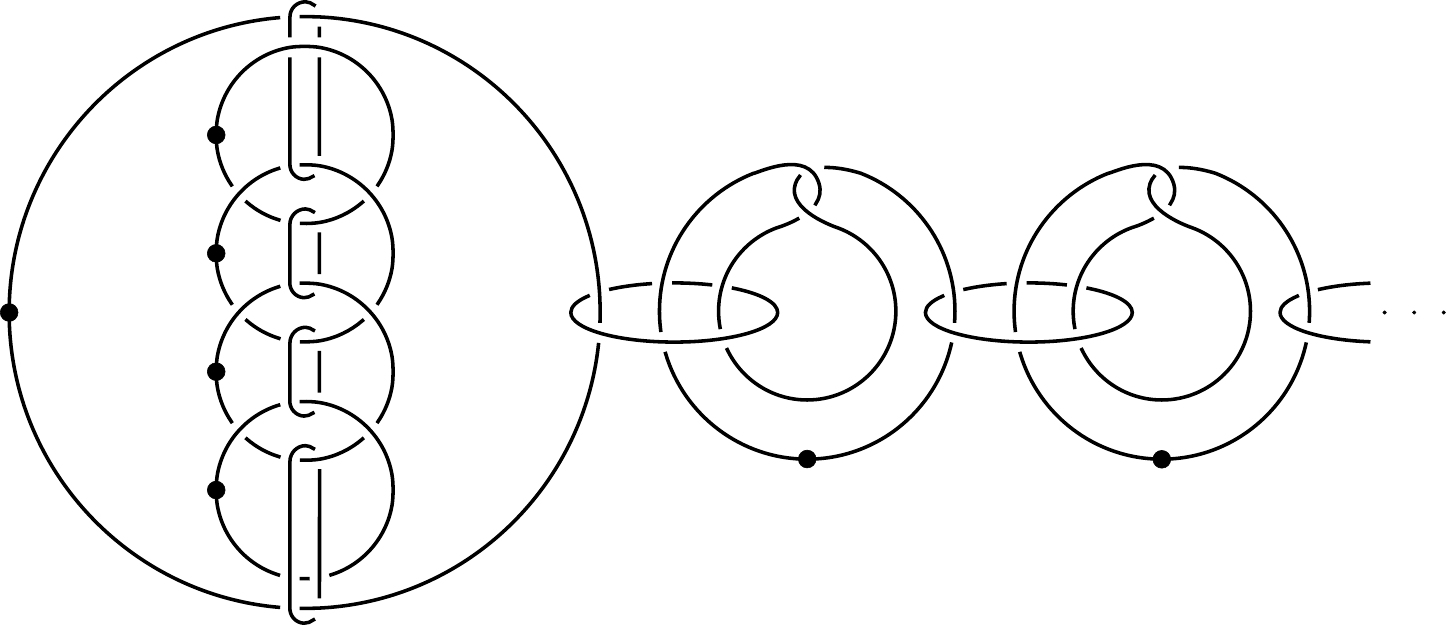}};
            \node at (-2.55,0.4) {$0$};
            \node at (-2.55,1.4) {$0$};
            \node at (-2.55,-0.4) {$0$};
            \node at (-2.55,-1.4) {$0$};
            \node at (-0.6,0.4) {$0$};
            \node at (1.8,0.4) {$0$};
            \node at (4.2,0.4) {$0$};
        \end{tikzpicture}
        \caption{A Kirby diagram for $\mcR_5$.}
        \label{fig:R51}
    \end{figure}

    To see how this diagram generalises for all $n\ge3$ and odd, we notice that the knot $T_{2,n}$ will differ from \cref{fig:T1Knots_a} by having $n$ half-twists, and so just like in \cref{fig:T1Ribbon_a}, we will need to do $n-1$ ribbon moves. As we saw in \cref{fig:T1Ribbon_b}, this will produce $n$ unlinked dotted circles, with $n-1$ $0$-framed $2$-handles as a result of the ribbon moves. Similar to \cref{fig:T1Ribbon_c}, this will yield the ribbon disc complement shown in \cref{fig:Rn1}. Therefore, the Kirby diagrams in \cref{fig:Rn1} for $n\ge3$ and odd represent a family of manifolds $\{\mcR_n\}_{n=3,5,7,...}$ obtained by attaching the Casson handle $CH^+$ to a ribbon disc complement of the knot $T_{2,n}\#T_{2,-n}$ and removing the boundary. It follows from \cite[Corollary 1.2]{elihomlidman} that $\{\mcR_n\}_{n=3,5,7,...}$ is a family of pairwise nondiffeomorphic exotic $\R^4$'s. 

    We now construct the second family $\{\mcR_n'\}_{n=3,5,7,...}$ by using different ribbon moves for the knot $T_{2,n}\#T_{2,-n}$. We will again start with the case of $n=5$. The knot in \cref{fig:T2Knots_a} is an equivalent diagram of the torus knot $T_{2,5}$. Similar to before, we form the connected sum $T_{2,5}\#T_{2,-5}$ as shown in \cref{fig:T2Knots_b} and \cref{fig:T2Knots_c}. 

    \begin{figure}[ht]
        \centering
        \begin{subfigure}[b]{0.3\textwidth}
            \centering
            \begin{tikzpicture}
                \node[anchor=center] at (0,0) {\includegraphics[scale=0.4]{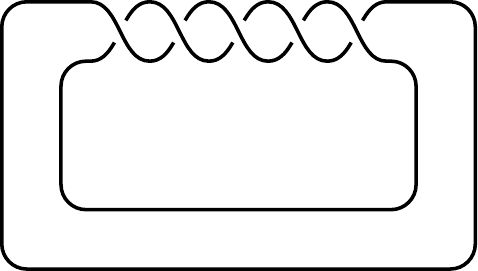}};
            \end{tikzpicture}
            \caption{}
            \label{fig:T2Knots_a}
        \end{subfigure}
        \begin{subfigure}[b]{0.5\textwidth}
            \centering
            \begin{tikzpicture}
                \node[anchor=center] at (0,0) {\includegraphics[scale=0.4]{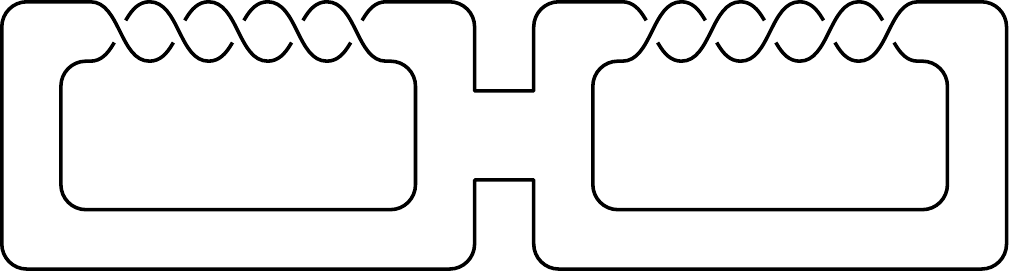}};
            \end{tikzpicture}
            \caption{}
            \label{fig:T2Knots_b}
        \end{subfigure}
        \begin{subfigure}[b]{0.4\textwidth}
            \centering
            \begin{tikzpicture}
                \node[anchor=center] at (0,0) {\includegraphics[scale=0.4]{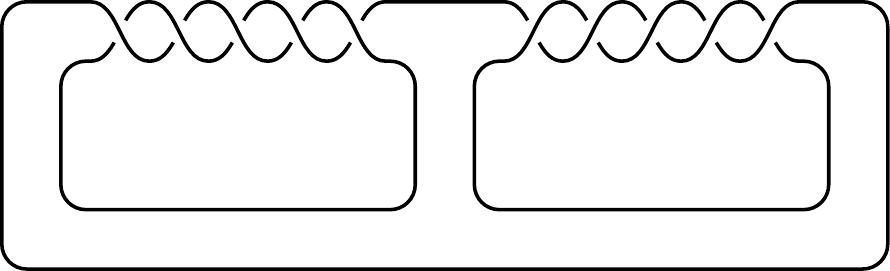}};
            \end{tikzpicture}
            \caption{}
            \label{fig:T2Knots_c}
        \end{subfigure}
        \caption{(a) The torus knot $T_{2,5}$. (b) and (c) The connected sum $T_{2,5}\#T_{2,-5}$.}
        \label{fig:T2Knots}
    \end{figure}

    To obtain a ribbon disc complement, we now do a single ribbon move as shown in \cref{fig:T2Ribbon_a} and \cref{fig:T2Ribbon_b}. \cref{fig:T2Ribbon_c} then follows from \cref{fig:T2Ribbon_b} by isotoping the bottom-most dotted strand to the top of the diagram. We can then unwind the top dotted circle from the bottom dotted circle in \cref{fig:T2Ribbon_c}. This results in the $2$-handle being wound around both dotted circles, as seen in \cref{fig:T2Ribbon_d}. Thus, \cref{fig:T2Ribbon_d} represents a ribbon disc complement of the knot $T_{2,5}\#T_{2,-5}$. 

    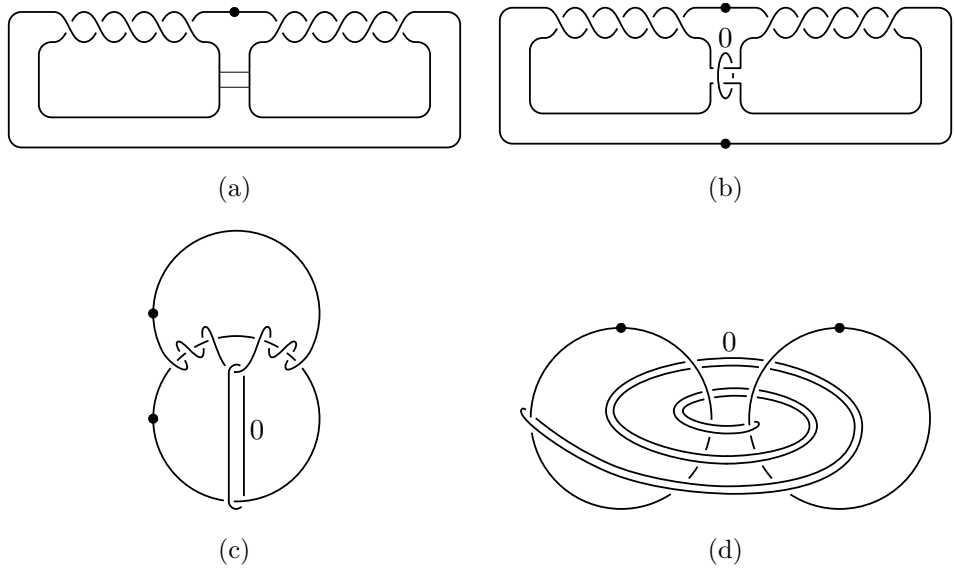
\begin{figure}[ht]
        \centering
        \begin{subfigure}[b]{0.4\textwidth}
            \centering
            \begin{tikzpicture}
                \node[anchor=center] at (0,0) {\includegraphics[scale=0.4]{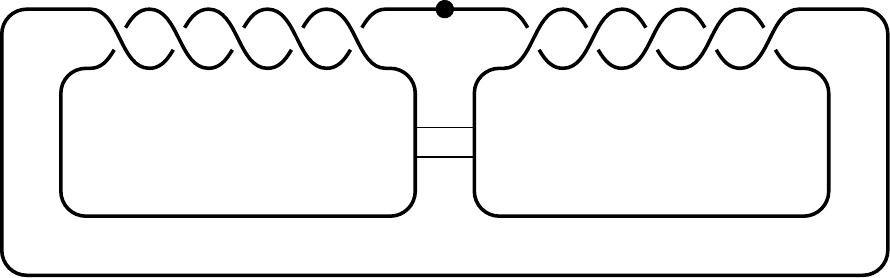}};
            \end{tikzpicture}
            \caption{}
            \label{fig:T2Ribbon_a}
        \end{subfigure}
        \begin{subfigure}[b]{0.4\textwidth}
            \centering
            \begin{tikzpicture}
                \node[anchor=center] at (0,0) {\includegraphics[scale=0.4]{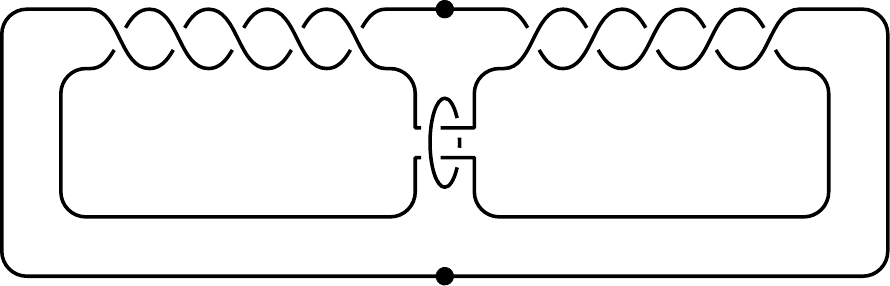}};
                \node at (0,0.5) {$0$};
            \end{tikzpicture}
            \caption{}
            \label{fig:T2Ribbon_b}
        \end{subfigure}
        \par\medskip
        \begin{subfigure}[b]{0.4\textwidth}
            \centering
            \begin{tikzpicture}
                \node[anchor=center] at (0,0) {\includegraphics[scale=0.4]{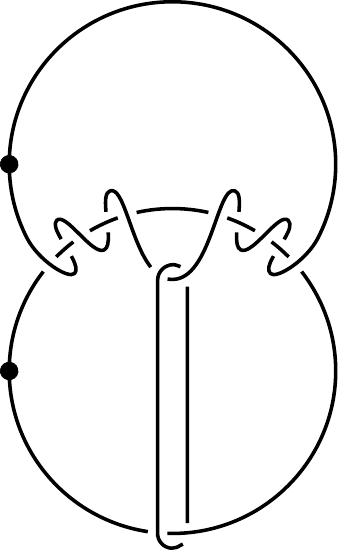}};
                \node at (0.3,-0.8) {$0$};
            \end{tikzpicture}
            \caption{}
            \label{fig:T2Ribbon_c}
        \end{subfigure}
        \begin{subfigure}[b]{0.4\textwidth}
            \centering
            \begin{tikzpicture}
                \node[anchor=center] at (0,0) {\includegraphics[scale=0.4]{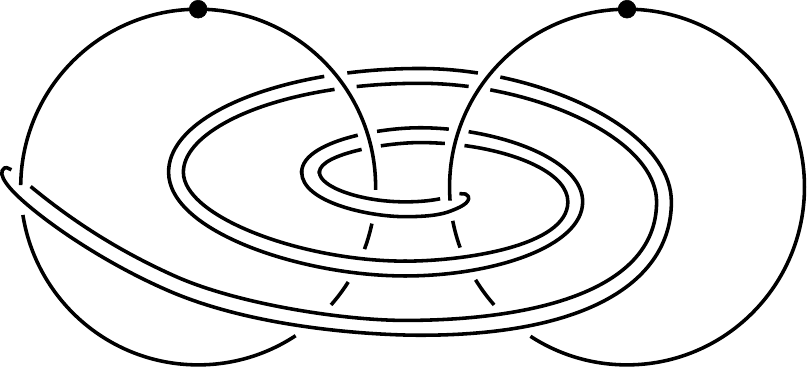}};
                \node at (0.05,1) {$0$};
            \end{tikzpicture}
            \caption{}
            \label{fig:T2Ribbon_d}
        \end{subfigure}
        \caption{Ribbon moves for a ribbon disc complement of the knot $T_{2,5}\#T_{2,-5}$.}
        \label{fig:T2Ribbon}
    \end{figure}

    As before, by \cite[Corollary 1.2]{elihomlidman}, the manifold $\mcR_5'$, obtained by attaching the Casson handle $CH^+$ to \cref{fig:T2Ribbon_d} and then removing the boundary, is an exotic $\R^4$ and is shown in \cref{fig:R52}.

    \begin{figure}[ht]
        \centering
        \begin{tikzpicture}
            \node[anchor=center] at (0,0) 
            {\includegraphics[scale=0.4]{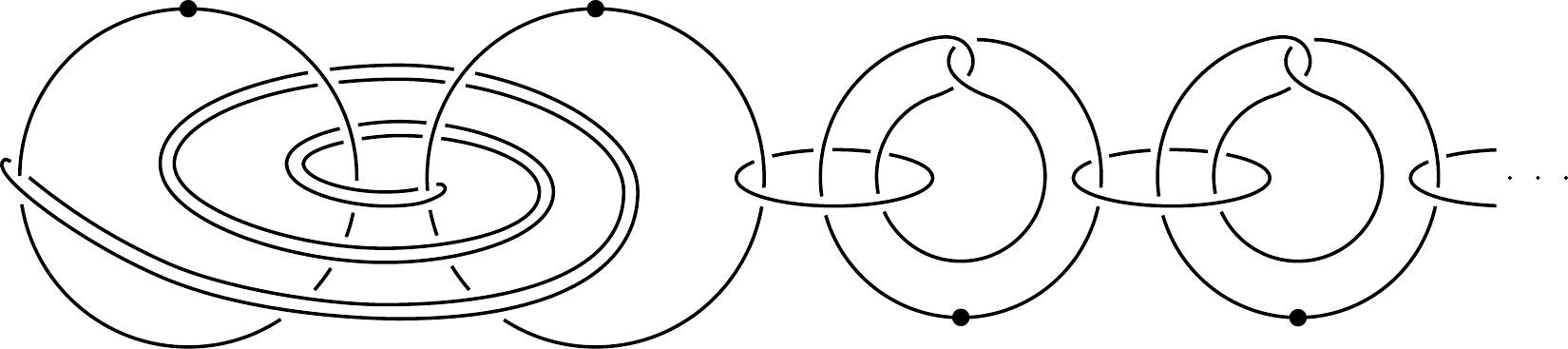}};
            \node at (-2.8,1) {$0$};
            \node at (0.1,0.4) {$0$};
            \node at (2.5,0.4) {$0$};
            \node at (4.9,0.4) {$0$};
        \end{tikzpicture}
        \caption{A Kirby diagram for $\mcR_5'$.}
        \label{fig:R52}
    \end{figure}

    To construct the Kirby diagram for all $\mcR_n'$ for $n\ge3$ and odd, we again notice that the torus knot $T_{2,n}$ will be the same as \cref{fig:T2Knots_a} except with $n$ half-twists. After forming the connected sum $T_{2,n}\#T_{2,-n}$ and performing the same single ribbon move, we will arrive at a diagram similar to \cref{fig:T2Ribbon_c} in which the top dotted circle would be wound $(n-1)/2$ times around the bottom dotted circle. After unwinding, we will get the ribbon disc complement shown in \cref{fig:Rn2}, in which the $2$-handle has a total of $(n-1)/2$ full windings around the inner strands of the two dotted circles. An alternative description of the diagram in \cref{fig:Rn2} is that the attaching circle of the $2$-handle has $2n+2$ crossings with the left dotted circle, and $2n$ crossings with the right dotted circle. The case of $n=3$ gives the same ribbon disc complement provided in \cite[Figure 12.35]{gompfstipsicz}. 
    We have now shown that the Kirby diagrams in \cref{fig:Rn2} for $n\ge3$ and odd represent a family of manifolds $\{\mcR_n'\}_{n=3,5,7,...}$ obtained by attaching the Casson handle $CH^+$ to a ribbon disc complement of the knot $T_{2,n}\#T_{2,-n}$. Similar to before, \cite[Corollary 1.2]{elihomlidman} implies that $\{\mcR_n'\}_{n=3,5,7,...}$ is a family of pairwise nondiffeomorphic exotic $\R^4$'s.
\end{proof}

\begin{remark}
    As mentioned in \cite{elihomlidman}, the results of \cref{thm:1} would still hold if we replaced the Casson handle $CH^+$ in \cref{fig:Rn1} and \cref{fig:Rn2} with the simplest negative Casson handle $CH^-$, which has a single negative self-plumbing at each stage. This means that we will have a negative Whitehead double instead of a positive Whitehead double at each stage of the Casson handle in \cref{fig:Rn1} and \cref{fig:Rn2}. We can also obtain an exotic $\R^4$ by attaching a Casson handle that has finitely many self-plumbings of one sign and no branching to any of the ribbon disc complements obtained above (see \cite{elihomlidman}).
    \label[remark]{rem:cassonhandles}
\end{remark}

\begin{proof}[Proof of \cref{thm:2}]
    We first show that the ribbon disc complements in \cref{fig:Rn1} and \cref{fig:Rn2} are diffeomorphic. To do this, we make use of the two moves shown in \cref{fig:Moves}. Both of these are sequences of Kirby moves and can be obtained using handle slides and cancellations. The first move, in \cref{fig:Move1}, is the result of sliding one of the $2$-handles over the other and then cancelling the $1$-handle. The second move, in \cref{fig:Move2}, is the result of sliding one of the $1$-handles over the other and then cancelling the $0$-framed $2$-handle.      

    \begin{figure}[ht]
        \centering
        \begin{subfigure}[b]{0.45\textwidth}
            \centering
            \begin{tikzpicture}
                \node[anchor=center] at (0,0) {\includegraphics[scale=0.4]{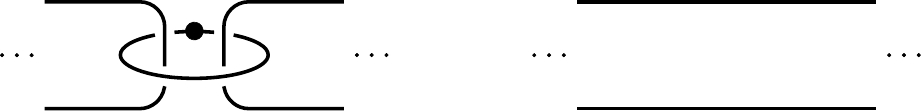}};
                \node at (-2.5,0.6) {$0$};
                \node at (-1.1,0.6) {$0$};
                \node at (1.8,0.6) {$0$};
                \draw[->] (-0.3,0) -- (0.3,0);
            \end{tikzpicture}
            \caption{}
            \label{fig:Move1}
        \end{subfigure}
        \begin{subfigure}[b]{0.45\textwidth}
            \centering
            \begin{tikzpicture}
                \node[anchor=center] at (0,0) {\includegraphics[scale=0.4]{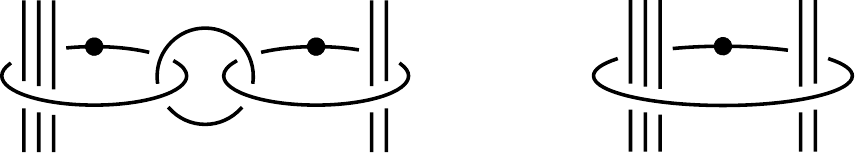}};
                \node at (-1.51,0.55) {$0$};
                \draw[->] (0.15,0) -- (0.85,0);
            \end{tikzpicture}
            \caption{}
            \label{fig:Move2}
        \end{subfigure}
        \caption{Diagrams of two moves that result from handle slides and cancellations. The vertical lines in \cref{fig:Move2} represent any $2$-handles that may be going through the dotted circles.} 
        \label{fig:Moves}
    \end{figure}
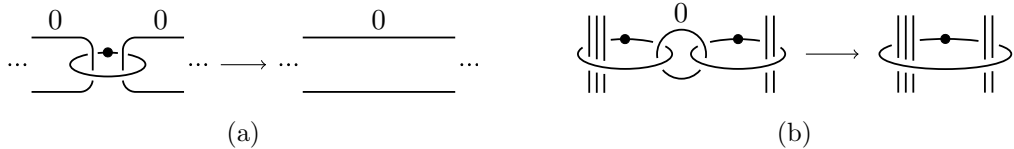
    
    As mentioned in \cref{rem:altdiagram}, we can represent $\mcR_n$ using the Kirby diagram in \cref{fig:Rn1Alt}. Thus, without the Casson handle attached, the ribbon disc complement in $\mcR_n$ can be represented by the diagram in \cref{fig:KirbyMoves1}. We first apply the move shown in \cref{fig:Move1} to the rightmost dotted circle in \cref{fig:KirbyMoves1}. This cancels the rightmost dotted circle and one of the adjacent $2$-handles, giving us \cref{fig:KirbyMoves2}. We then successively apply the move shown in \cref{fig:Move2} to each of the $2$-handles in \cref{fig:KirbyMoves2} except for the rightmost $2$-handle. This can be done, for example, in a counter-clockwise order in \cref{fig:KirbyMoves2}, and by applying isotopies to obtain the configuration of \cref{fig:Move2} at each stage. Doing this leaves us with only one $2$-handle as shown in \cref{fig:KirbyMoves3}. 

    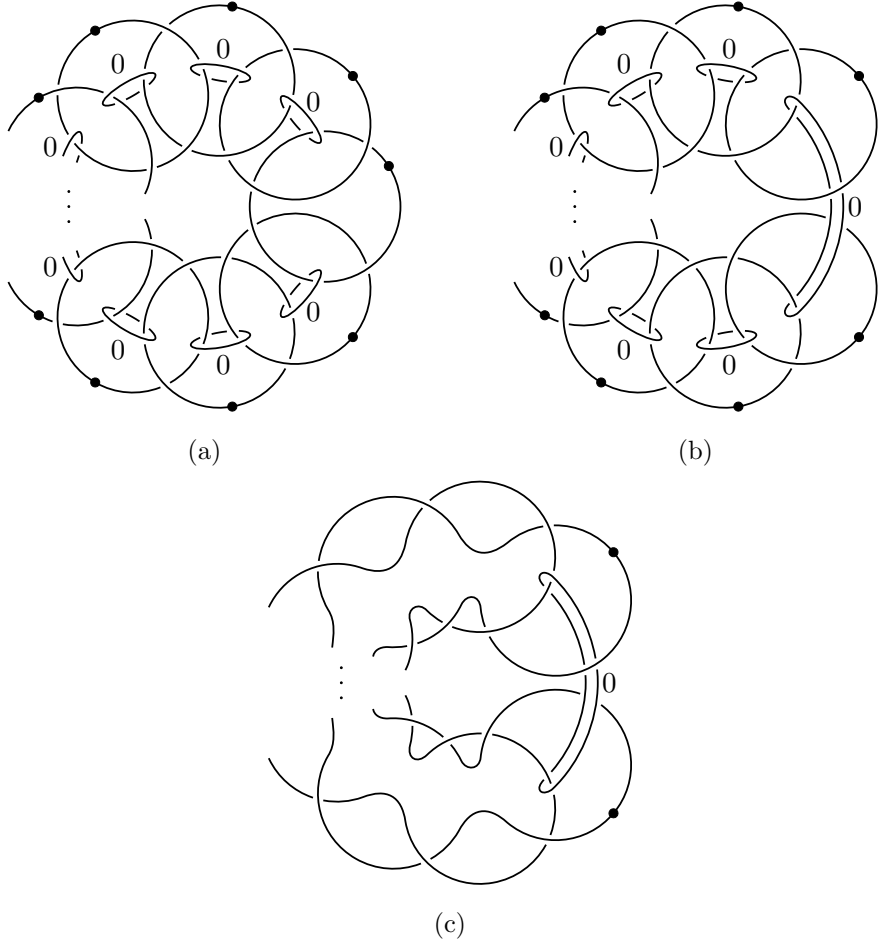
\begin{figure}[ht]
        \centering
        \begin{subfigure}[b]{0.4\textwidth}
            \centering
            \begin{tikzpicture}
                \node[anchor=center] at (0,0) {\includegraphics[scale=0.4]{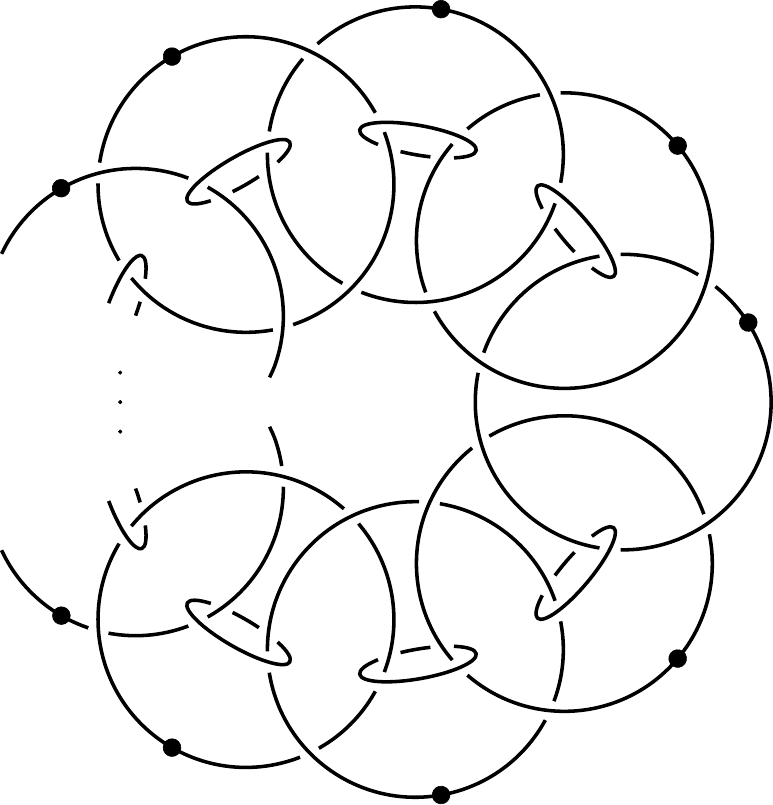}};
                \node at (-1.4+2.85,1.4) {$0$};
                \node at (-2.6+2.85,2.1) {$0$};
                \node at (-4+2.85,1.9) {$0$};
                \node at (-4.9+2.85,0.8) {$0$};
                \node at (-1.4+2.85,-1.4) {$0$};
                \node at (-2.6+2.85,-2.1) {$0$};
                \node at (-4+2.85,-1.9) {$0$};
                \node at (-4.9+2.85,-0.8) {$0$};
            \end{tikzpicture}
            \caption{}
            \label{fig:KirbyMoves1}
        \end{subfigure}
        \begin{subfigure}[b]{0.4\textwidth}
            \centering
            \begin{tikzpicture}
                \node[anchor=center] at (0,0) {\includegraphics[scale=0.4]{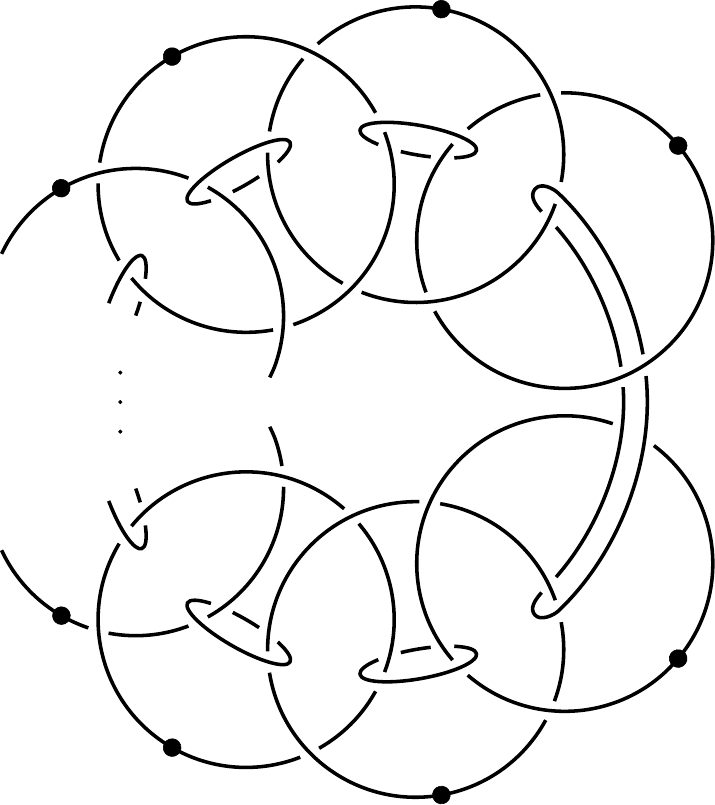}};
                \node at (0.45,2.1) {$0$};
                \node at (-0.95,1.9) {$0$};
                \node at (-1.85,0.8) {$0$};
                \node at (0.45,-2.1) {$0$};
                \node at (-0.95,-1.9) {$0$};
                \node at (-1.85,-0.8) {$0$};
                \node at (2.12,0) {$0$};
            \end{tikzpicture}
            \caption{}
            \label{fig:KirbyMoves2}
        \end{subfigure}
        \begin{subfigure}[b]{0.4\textwidth}
            \centering
            \begin{tikzpicture}
                \node[anchor=center] at (0,0) {\includegraphics[scale=0.4]{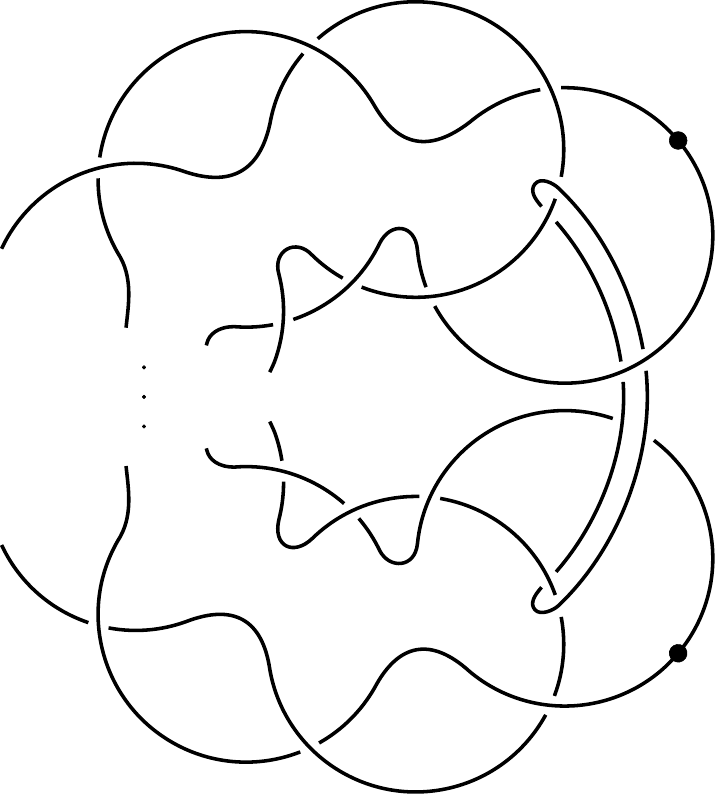}};
                \node at (2.12,0) {$0$};
            \end{tikzpicture}
            \caption{}
            \label{fig:KirbyMoves3}
        \end{subfigure}
        \caption{The result of Kirby moves, going from the disc complement of $\mcR_n$ to the disc complement of $\mcR_n'$. The diagram in (a) is the same as the disc complement in \cref{fig:Rn1Alt}, and contains $n$ dotted circles.}
        \label{fig:KirbyMoves}
    \end{figure}
    
    \cref{fig:Same} shows some equivalent ways to draw the Kirby diagram in \cref{fig:KirbyMoves3}. We have drawn the inner crossings of \cref{fig:KirbyMoves3} as the twists on the right side of \cref{fig:KirbyMoves4}, and the outer crossings as the twists on the left side. There are $n-2$ half-twists on both the left and the right side of \cref{fig:KirbyMoves4}, since there were $n-1$ dotted circles in \cref{fig:KirbyMoves2}. By bringing the $2$-handle to the inside of the diagram, we then obtain \cref{fig:KirbyMoves5}. As a result, there are now $n$ half-twists on both sides of the $2$-handle in \cref{fig:KirbyMoves5}. 

    \cref{fig:KirbyMoves5} is the general form of \cref{fig:T2Ribbon_c} for all $n\ge3$ and odd. Thus, in the same way that we went from \cref{fig:T2Ribbon_c} to \cref{fig:T2Ribbon_d}, we can now unwind the top dotted circle from the bottom dotted circle in \cref{fig:KirbyMoves5}, which gives us exactly the disc complement in $\mcR_n'$. That is, we obtain the Kirby diagram in \cref{fig:Rn2} but without the Casson handle attached. Since there is a sequence of Kirby moves connecting them, we conclude that the disc complements in $\mcR_n$ and $\mcR_n'$ are diffeomorphic. 

    \begin{figure}[ht]
        \centering
        \begin{subfigure}[b]{0.3\textwidth}
            \centering
            \begin{tikzpicture}
                \node[anchor=center] at (0,0) {\includegraphics[scale=0.4]{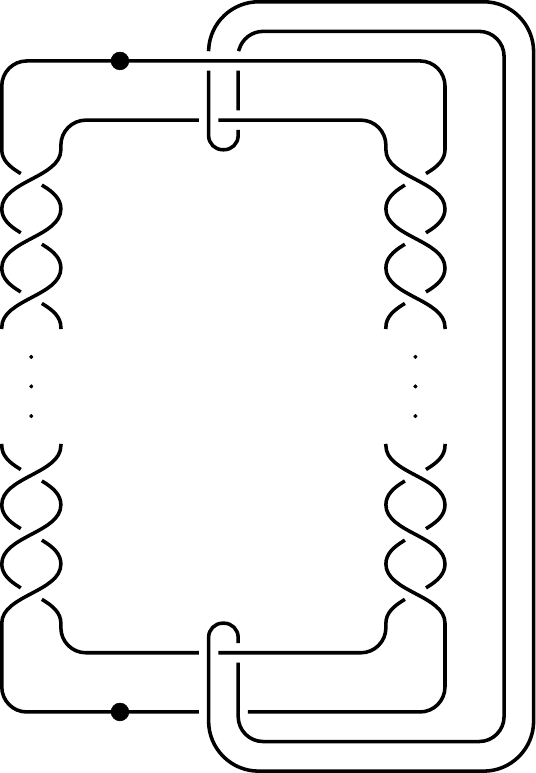}};
                \node at (-0.3,-1.3) {$0$};
            \end{tikzpicture}
            \caption{}
            \label{fig:KirbyMoves4}
        \end{subfigure}
        \begin{subfigure}[b]{0.6\textwidth}
            \centering
            \begin{tikzpicture}
                \node[anchor=center] at (0,0) {\includegraphics[scale=0.4]{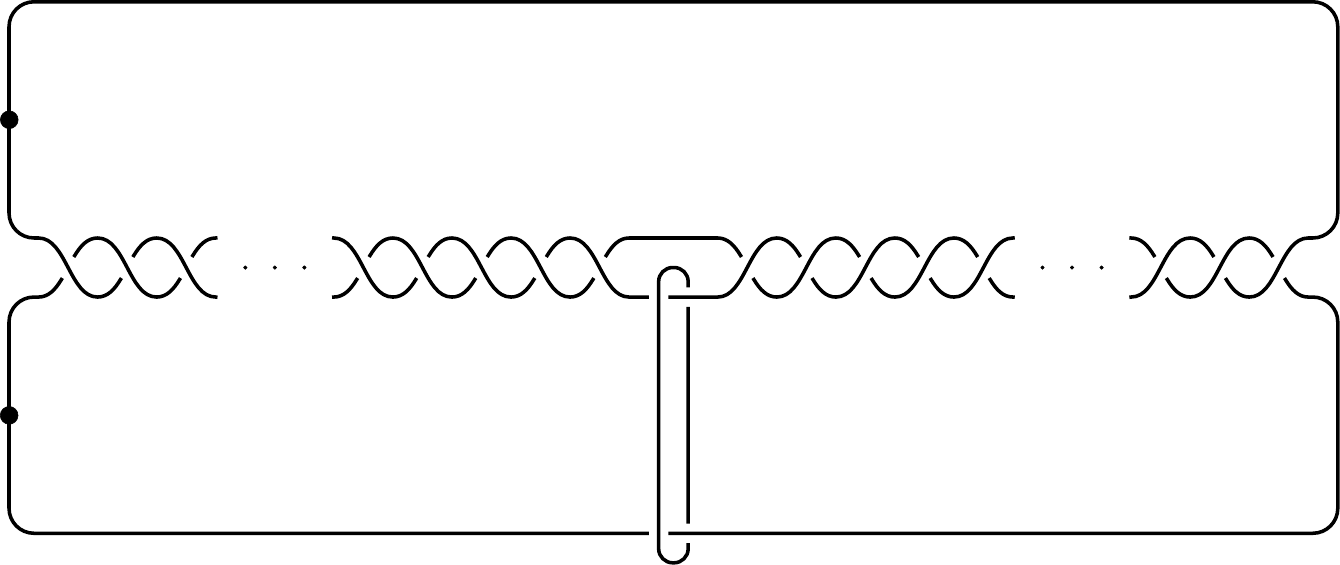}};
                \node at (0.3,-0.9) {$0$};
            \end{tikzpicture}
            \caption{}
            \label{fig:KirbyMoves5}
        \end{subfigure}
        \caption{Equivalent diagrams to \cref{fig:KirbyMoves3}. In (a), there are $n-2$ half-twists on both sides. In (b), there are $n$ half-twists on both sides.}
        \label{fig:Same}
    \end{figure}

    We now consider the attachment of the Casson handle. From \cref{fig:KirbyMoves}, we can see that a meridian of the ribbon disc in \cref{fig:KirbyMoves1} is mapped to a meridian of the ribbon disc in \cref{fig:KirbyMoves3} while staying in $S^3$ during the Kirby moves. Since the Casson handle can be attached to any meridian of the ribbon discs, we can extend the diffeomorphism that we obtained above between the disc complements to a diffeomorphism between $\mcR_n$ and $\mcR_n'$. Thus, for all $n\ge3$ and odd, $\mcR_n$ and $\mcR_n'$ are diffeomorphic. 
\end{proof}

\begin{proof}[Proof of \cref{thm:3}]
    We first show that \cref{fig:Rnk} can be obtained by attaching the Casson handle $CH^+$ to a ribbon disc complement of the three stranded pretzel knot $P(n,-n,2k)$. We will explicitly construct $\mcR_{5,k}$ and then generalise it for all $n\ge3$ and odd. The pretzel knot $P(5,-5,2k)$ is shown in \cref{fig:P1}. The box in \cref{fig:P1} represents $k$ full twists, which is equivalent to $2k$ half-twists as indicated in the notation $P(5,-5,2k)$. We will take $k$ full twists to mean $k$ positive full twists when $k>0$, and $|k|$ negative full twists when $k<0$. 

    Similar to the construction of $\mcR_n'$, we can obtain the Kirby diagram of a ribbon disc complement of $P(5,-5,2k)$ by doing a single ribbon move as shown in \cref{fig:P2} and \cref{fig:P3}. We then arrive at \cref{fig:P4} by moving the bottom-most dotted strand to the top in \cref{fig:P3}. This is the same as \cref{fig:T2Ribbon_c} but with $k$ full twists in the $2$-handle. Thus by unwinding and separating the two dotted circles, we will get the disc complement in \cref{fig:Rnk} for the $n=5$ case. 

    In the general case for $n\ge3$ and odd, the knot $P(n,-n,2k)$ will differ from \cref{fig:P1} by having $n$ positive half-twists on the left side and $n$ negative half-twists on the right side. By doing the same single ribbon move, we will obtain a diagram similar to \cref{fig:P4} where the top dotted circle has $(n-1)/2$ total windings around the bottom dotted circle. After unwinding, we will obtain the Kirby diagram for the disc complement shown in \cref{fig:Rnk}. Thus, $\mcR_{n,k}$ is obtained by attaching the Casson handle $CH^+$ to a ribbon disc complement of the knot $P(n,-n,2k)$ for $n\ge3$ and odd and $k\in\Z$, and then removing the boundary. Therefore, \cite[Theorem 1.1(1)]{elihomlidman} implies that $\mcR_{n,k}$ is an exotic $\R^4$ for all $n\ge3$ and odd, and $k\in\Z$.

    \begin{figure}[ht]
        \centering
        \begin{subfigure}[b]{0.4\textwidth}
            \centering
            \begin{tikzpicture}
                \node[anchor=center] at (0,0) {\includegraphics[scale=0.4]{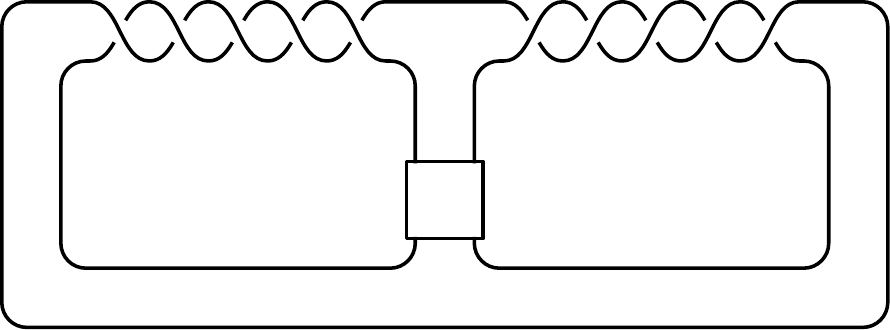}};
                \node at (0,-0.24) {$k$};
            \end{tikzpicture}
            \caption{}
            \label{fig:P1}
        \end{subfigure}
        \begin{subfigure}[b]{0.4\textwidth}
            \centering
            \begin{tikzpicture}
                \node[anchor=center] at (0,0) {\includegraphics[scale=0.4]{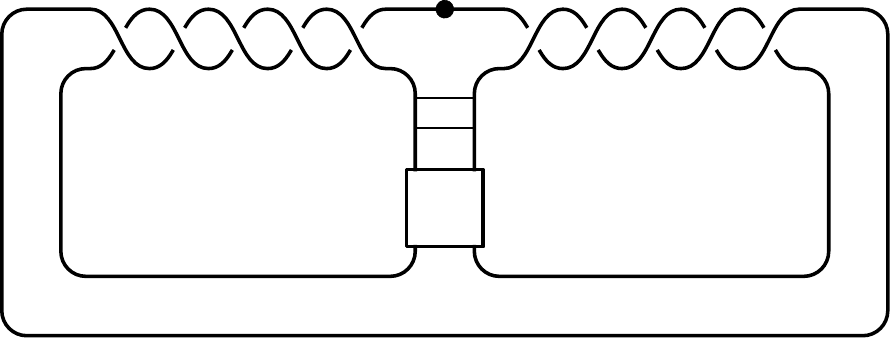}};
                \node at (0,-0.26) {$k$};
            \end{tikzpicture}
            \caption{}
            \label{fig:P2}
        \end{subfigure}
        \par\medskip
        \begin{subfigure}[b]{0.4\textwidth}
            \centering
            \begin{tikzpicture}
                \node[anchor=center] at (0,0) {\includegraphics[scale=0.4]{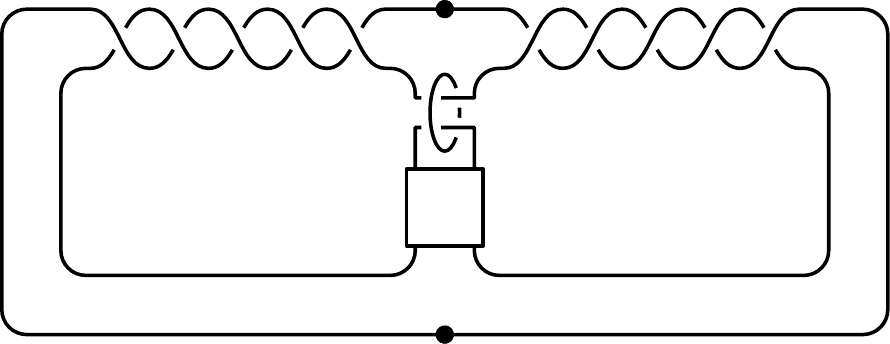}};
                \node at (0,-0.26) {$k$};
                \node at (0.35,0.4) {$0$};
            \end{tikzpicture}
            \caption{}
            \label{fig:P3}
        \end{subfigure}
        \begin{subfigure}[b]{0.4\textwidth}
            \centering
            \begin{tikzpicture}
                \node[anchor=center] at (0,0) {\includegraphics[scale=0.4]{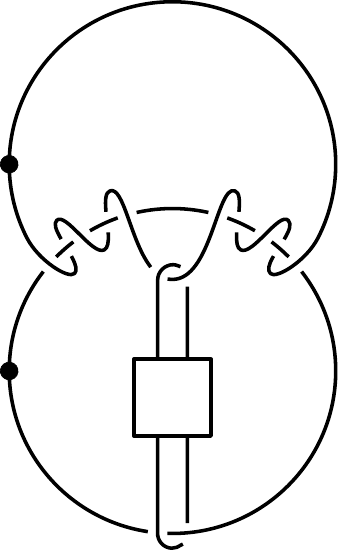}};
                \node at (0.03,-0.83) {$k$};
                \node at (0.3,-0.3) {$0$};
            \end{tikzpicture}
            \caption{}
            \label{fig:P4}
        \end{subfigure}
        \caption{Ribbon moves for a ribbon disc complement of the knot $P(5,-5,2k)$.}
        \label{fig:P}
    \end{figure}
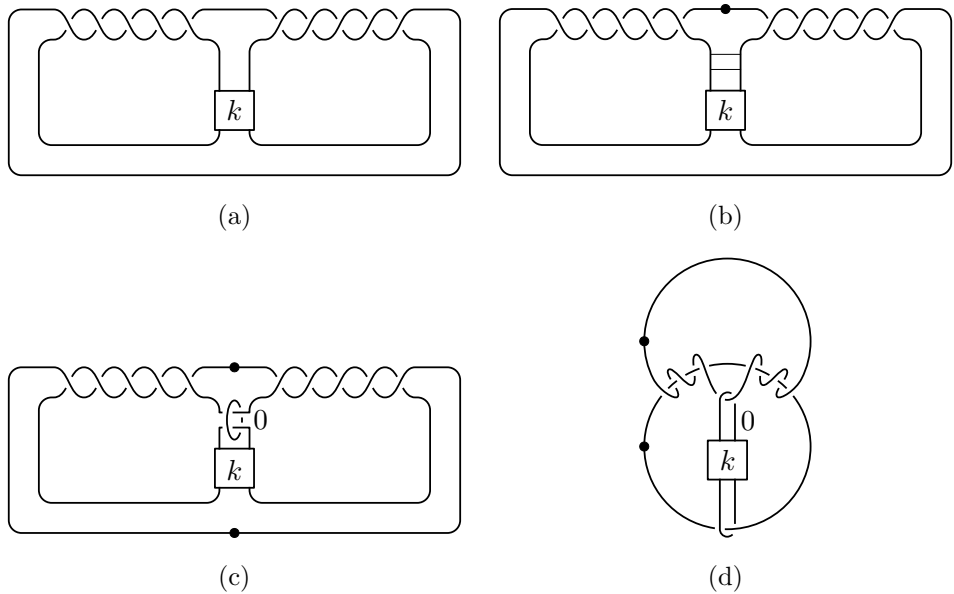
    
    To prove the second part of the theorem, we first show that the knot Floer homology $\widehat{HFK}$ of the knot $P(n,-n,2k)$ for $n\ge 3$ and odd, and $k\in\Z$, does not depend on $k$. We use a result of Hedden and Watson, who showed that for a non-trivial knot $K$ that is obtained as a band sum of two unknots, and a knot $K_i$ obtained by adding $i$ full twists to the band, the knot Floer homologies $\widehat{HFK}(K)$ and $\widehat{HFK}(K_i)$ are identical \cite[Theorem 1]{heddenwatson}. The knot $T_{2,n}\#T_{2,-n}$, which is the same as $P(n,-n,0)$, can be drawn as a band sum of two unknots, as shown in \cref{fig:HFK1} for the case $n=5$. This is because by unwinding the band, as shown in \cref{fig:HFK2}, we get the same diagram as \cref{fig:T2Knots_c}. In general, for $n\ge3$ and odd, we can obtain a similar diagram to \cref{fig:HFK1} for $T_{2,n}\#T_{2,-n}$ in which the band has $(n-1)/2$ full windings around the two unknots. By comparing \cref{fig:P1} with \cref{fig:HFK2}, we can see that the knot $P(n,-n,2k)$ is obtained by adding $k$ full twists to the band in $T_{2,n}\#T_{2,-n}$. Thus, by \cite[Theorem 1]{heddenwatson}, we see that $\widehat{HFK}(P(n,-n,2k))\cong \widehat{HFK}(P(n,-n,2k'))$, for all $k, k'\in\Z$. We are grateful to Jennifer Hom for suggesting the use of the results of \cite{heddenwatson} in the above argument. 

    \begin{figure}[ht]
        \centering
        \begin{subfigure}[b]{0.4\textwidth}
            \centering
            \begin{tikzpicture}
                \node[anchor=center] at (0,0) {\includegraphics[scale=0.4]{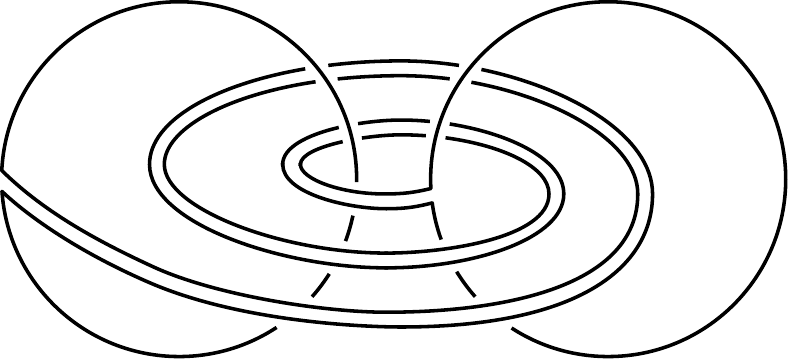}};
            \end{tikzpicture}
            \caption{}
            \label{fig:HFK1}
        \end{subfigure}
        \begin{subfigure}[b]{0.4\textwidth}
            \centering
            \begin{tikzpicture}
                \node[anchor=center] at (0,0) {\includegraphics[scale=0.4]{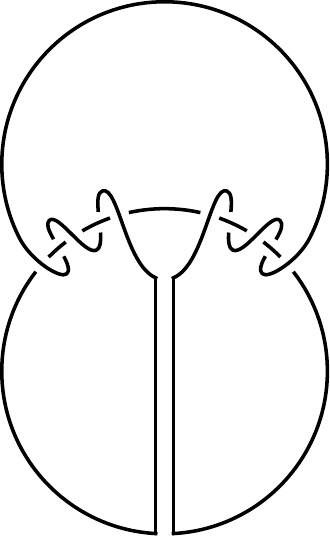}};
            \end{tikzpicture}
            \caption{}
            \label{fig:HFK2}
        \end{subfigure}
        \caption{(a) Diagram of $T_{2,5}\#T_{2,-5}$ as the band sum of two unknots. (b) An equivalent diagram to (a) which can be identified to be the same as \cref{fig:T2Knots_c}.}
        \label{fig:HFK}
    \end{figure}  

    Eli, Hom, and Lidman showed that if the knot Floer homology $\widehat{HFK}_{\text{red}}$ of two nontrivial slice knots has different maximal nontrivial Maslov gradings, then the exotic $\R^4$'s built on them by attaching $CH^+$ to a slice disc complement and removing the boundary are not diffeomorphic \cite[Theorem 1.1]{elihomlidman}. As mentioned in \cite{elihomlidman}, the maximal nontrivial Maslov grading of $\widehat{HFK}_{\text{red}}$ is the same as that of $\widehat{HFK}$ for nontrivial, thin, slice knots. The knot $T_{2,n}\#T_{2,-n}$ is thin because it is alternating \cite{ozsvathszabo2}. If $k>0$, then the knot $P(n,-n,2k)$ is thin for all $n\ge3$ and odd, as shown in \cite{varvarezos}. In the case $k<0$, we note that $P(n,-n,2k)$ is the mirror of $P(n,-n,-2k)$. Therefore since $P(n,-n,-2k)$ is thin, the knot $P(n,-n,2k)$ must also be thin (see \cite{ozsvathszabo1} and \cite[Proposition 7.1.2]{ozsvathstipsiczszabo}). It follows that $P(n,-n,2k)$ is thin for all $n\ge3$ and odd, and $k\in\Z$. This implies that its maximal nontrivial Maslov grading is the same for $\widehat{HFK}_{\text{red}}$ and $\widehat{HFK}$. Thus, since $P(n,-n,2k)$ has the same knot Floer homology $\widehat{HFK}$ for all $k\in\Z$, it follows that it has the same maximal nontrivial Maslov grading of $\widehat{HFK}_{\text{red}}$ for all $k\in\Z$. 

    As given in \cite[Remark 6.2]{elihomlidman}, the maximal nontrivial Maslov grading of $T_{2,n}\#T_{2,-n}$, which is the same knot as $P(n,-n,0)$, is $n-1$. Therefore, the maximal nontrivial Maslov grading of the knot $P(n,-n,2k)$ is also $n-1$. We can now use the results of Eli, Hom, and Lidman \cite[Theorem 1.1(2)]{elihomlidman} to conclude that if $n\ne n'$, then $\mcR_{n,k}$ and $\mcR_{n',k'}$ are not diffeomorphic for all $k,k'\in\Z$, since the knot Floer homology $\widehat{HFK}_{\text{red}}$ of the knots $P(n,-n,2k)$ and $P(n',-n',2k)$ has different maximal nontrivial Maslov gradings.
\end{proof}
\vspace{1pt}

\bibliographystyle{amsalpha}
\bibliography{bib}
\end{document}